\documentclass[reqno]{siamltex}

\usepackage{graphicx}
\usepackage{fullpage}
\usepackage{hyperref}
\usepackage{color}
\usepackage{amsmath, amssymb}

\newcommand{\DS}{\displaystyle}
\newcommand {\R} {{\mathbb{R}}}

\newcommand {\Rn} {\mathbb{R}^n}
\newcommand {\Rl} {\mathbb{R}^l}
\newcommand {\be}{{\beta_\varepsilon}}

\newcommand {\Rp} {\mathbb{R}^+}

\newcommand {\co} {\mathcal{O}}
\newcommand {\ob}{\overline{\mathcal{O}}}
\newcommand {\UC}{U\!C(\Rn)}

\newcommand {\C} {\mathcal{C}}
\newcommand {\A} {\mathcal{A}}
\newcommand {\F} {\mathcal{F}}
\newcommand {\M} {\mathcal{M}}
\newcommand {\Ell} {\mathcal{L}}

\newcommand {\E} {\mathbb{E}}

\newcommand {\ve} {\varepsilon}
\newcommand {\Prob} {\mathbb{P}}
\newcommand {\toinf} {{\to \infty}}
\newcommand {\vf}{\varphi}
\newcommand {\bF}{\mathbb{F}}
\newcommand {\veps}{{v^\varepsilon}}
\newcommand {\vek}{{v^{\varepsilon_k}}}
\newcommand {\bek}{{\beta_{\varepsilon_k}}}
\newcommand {\gek}{{g^{\varepsilon_k}}}

\DeclareMathOperator{\ess}{ess}
\DeclareMathOperator{\loc}{loc}

\DeclareMathOperator{\tr}{tr}
\DeclareMathOperator{\dist}{dist}
\DeclareMathOperator{\diam}{diam}

\newtheorem{remark}{Remark}

\numberwithin{equation}{section}

\begin{document}

\author{Mark H. A. Davis\thanks{Department of Mathematics, Imperial College London, London SW7 2AZ, UK. Email address: \texttt{mark.davis@imperial.ac.uk}}
\and Xin Guo\thanks{Department of Industrial Engineering
and Operations Research, University of California at Berkeley, CA 94720-1777. Email address: \texttt{xinguo@ieor.berkeley.edu}}  \and Guoliang Wu\thanks{Department of Mathematics, University of California at Berkeley, CA 94720-3840. Email address: \texttt{guoliang@math.berkeley.edu}}}

\title{Impulse Control of Multidimensional Jump Diffusions}

\maketitle
\begin{abstract} This paper studies  regularity property of the value function
for an infinite-horizon discounted cost impulse control problem,  where the
underlying controlled process is a multidimensional jump diffusion with possibly
`infinite-activity' jumps. Surprisingly, despite these jumps,
 we  obtain the same degree of regularity as for the diffusion case, at least when the jump satisfies certain integrability conditions.
\end{abstract}

\section{Introduction}

This paper is concerned with regularity of the value function in an
impulse control problem for an $n$-dimensional jump diffusion
process $X(t)$.

In the absence of control,  the stochastic process
$X(t)$ is governed by the following SDE:
\begin{equation} \label{Eq. dynamics}
dX(t)=\mu(X(t^-)) dt+\sigma(X(t^-)) dW(t)+\int_{\Rl}
j(X(t^-),z)\widetilde N(dt, dz), \ \ X(0)=x.
\end{equation}
Here $W(t)$  is an $m$-dimensional Brownian motion and $N(\cdot,\cdot)$
 a \emph{Poisson random measure}
 on $\Rp\times\Rl$, with $W$ and $N$ independent.  The \emph{L\'evy measure}
$\nu(\cdot):=\E(N(1,\cdot))$ may be unbounded and $\widetilde N(dt, dz)$ is its
\emph{compensated Poisson random measure} with
$\widetilde N(dt, dz):=N(dt, dz)-\nu(dz)dt.$
The parameters $b, \sigma, j$ satisfy appropriate
conditions (see Section \ref{sec assumptions}) to ensure
the well-definedness of this SDE.

If an admissible control policy
$V=(\tau_1,\xi_1; \tau_2,\xi_2; \ldots)$ is adopted, then $X(t)$
evolves as
\begin{align}
  dX(t)=\mu(X(t^-)) dt+\sigma(X(t^-)) dW(t)+\int_{\Rl} j(X(t^-),z)\widetilde N(dt, dz)+ \sum_i
  \delta(t-\tau_i)\xi_i,
  \label{eqn controlled jump}
\end{align}
where $\delta(\cdot)$ denotes the Dirac delta function. With this given
control, the associated
total expected cost (objective function) is
\begin{align}
  J_x[V]:=\E_x\left(\int_0^\infty e^{-rt}f(X(t))dt+\sum_{i=1}^\infty e^{-r\tau_i}B(\xi_i)\right).
  \label{def of J (total cost)}
\end{align}
  The aim is to minimize the total
cost over all admissible control policies, with the value function
\begin{align}
  u(x)=\inf_{V} J_x[V].
  \label{def of u (value function)}
\end{align}

\paragraph{HJB and Regularity}

 A heuristic derivation from the Dynamic Programming Principle shows that
the value function (\ref{def of u (value function)}) is associated with the following Quasi-Variational-Inequality, or
HJB, by
\begin{equation}\label{HJB}\tag{HJB}
  \max(\Ell u-f, u-\M u)=0 \quad\text{ in }\Rn.
\end{equation}
where $\M\varphi(x)$ is the so called minimal operator such that
\begin{equation}
\M\varphi(x)=\inf_{\xi\in\Rn}(\varphi(x+\xi)+B(\xi)),\label{def of M}
\end{equation}
and $\Ell \varphi(x)$ is the partial integro-differential operator
\begin{equation}
\Ell \varphi(x)=
- \tr \left[A\cdot D^2\varphi(x)\right]- \mu(x)\cdot D\varphi(x)+r \vf(x) +
\int_{\Rl} \left[\varphi(x+j(x, z))-\vf(x)-j(x,z)D\varphi(x)\right]\nu(dz)\label{def of Ell},
\end{equation}
where the matrix $A$ is given by $A=\left(a_{ij}\right)_{n\times n}=\frac 1 2 \sigma(x)\sigma(x)^\textsf{T}$.
Most recently   \cite{Seydel09} proves rigorously  that
indeed the value function is a continuous solution to (\ref{HJB}) in a viscosity sense.

Nevertheless, an important question remains: under what conditions is the value function a solution to the
(\ref{HJB}) in a classical sense? Or, what is the degree of the smoothness (i.e. regularity property) for the value function in general?
This is the focus of our paper.

Regularity property has been  one of the central  topics in PDEs theory \cite{LU68, GT98, Lieberman96}. Besides its obvious and natural theoretical interest, regularity
study provides useful insight for numerical approximation.
 Controls of the impulse type by allowing discrete state space and  fixed cost proves most desirable for application purpose. See
\cite{CR78, HT83, HST83, Taksar85}
for risk
management, \cite{TH90, MT94} for real options, \cite{BP00, Korn98,
Korn99, EH88, MP95, OS02} for transaction cost in portfolio
management, \cite{JS95, CCTZ06} for insurance models, \cite{VMP07, BP09}
for liquidity risk, and \cite{Jeanblanc93, MO98, CZ99} for optimal control
of exchange rates.
Meanwhile,  jump diffusions such as L\'evy processes, have been very popular in financial modeling. See for example \cite{Pham97, Kallsen00, Eberlein01, CGMY03, CT04, YJB06, LM08}.

Combined, there is a growing interest and need to analyze impulse controls on jump diffusions. Unfortunately, impulse control is among the hardest to analyze
and the regularity study for the associated HJB or the value function is largely
open, except for some special and degenerate cases such as
 singular control and optimal stopping problems, see \cite{Ma92, Pham04, GT07, BX09}.
 One of the difficulty in establishing the regularity property
lie in the non-linear, non-local operator $\M u$ in Eqn (\ref{HJB}),
 Another difficulty is the partial integro-differential operator  $\Ell \varphi(x)$
associated with the jump processes.  For the special case when the
controlled diffusion is without jumps,  
\cite{BL82} established the regularity property
by assuming that the control is bounded and non-negative with additional smoothness in the cost structure.  Recently, 
\cite{GW09} applies the tricks of  translating the
regularity of the minimal operator in the action
region into that of the PDEs in the continuation region.   However, all these technique fail for a controlled  jump
diffusion component. The  major issue is the additional partial integro-differential operator. Moreover, (possibly infinite) jumps through the
boundary might potentially reduces the degree of smoothness for  the value
function.

\paragraph{Our work} This paper investigates the regularity of the value functions for
the jump-diffusion models with impulse control.  Building on the existence  of the viscosity solution to the HJB for the value function \cite{Seydel09} and the trick of
\cite{GW09} for the non-local minimal operator,  we focus on the
partial integro-differential operator in the continuation region. There are two distinct cases:  when the jump is
driven by compound Poisson process, or equivalently when the L\'evy
measure is finite, the analysis is fairly straightforward by
the standard Schauder's estimate from PDEs, as in \cite{GW09}. For the most interesting
case of infinite L\'evy measure, the key is to  combine the classical
$L^p$ theory with the ``bootstrap'' argument to obtain regularity of
the partial integro-differential operator. Finally, to deal with the regularity along the ``free boundary'', appropriate penalty function is devised. Surprisingly, despite the added possibly infinite jumps,
 we have here the same  regularity as in the diffusion case, at least when the jump satisfies certain integrability conditions.

\section{Assumptions and Notations}\label{sec assumptions}
We first specify the exact mathematical framework for our problem. Given a filtered and complete probability space $(\Omega, \F,
(\F_t)_{t\ge0}, \Prob)$ satisfying the usual conditions, we have a controlled jump diffusion process $X(t)$ as defined above at (\ref{eqn controlled jump}).
An \emph{admissible impulse control} $V$ consists of a sequence of
stopping times $\tau_1, \tau_2,\ldots$ with respect to $\F_t$ and a corresponding sequence
of $\Rn$-valued random variables $\xi_1, \xi_2,\ldots$ satisfying
the conditions
\begin{align*}
  \begin{cases}
    0< \tau_1< \tau_2<\cdots< \tau_i< \ldots, \\
    \tau_i\toinf \text{ a.s. as }i\toinf, \\
    \xi_i\in\F_{\tau_i}, \quad\forall i\ge 1.
  \end{cases}
\end{align*}
The associated total expected cost (objective function) is given by (\ref{def of u (value function)})
where  $f$ is the ``running cost'', $B$ is the ``transaction cost'' and $r>0$ is the discount factor. We assume that all the randomness comes from $W$ and $N$, so that
the filtration $\bF=(\F_t)_{t\ge 0}$ is generated by $W$ and
$N$.

We next  specify detailed conditions on the coefficients to ensure
the existence and uniqueness of \eqref{Eq. dynamics}, as well as the
conditions on $f$ and $B$ in \S\ref{sec assumptions}.

 Throughout this paper, we shall
impose the following standing assumptions:
\begin{itemize}
  \item [\textbf{(A1)}] Lipschitz conditions on $\mu: \Rn\to\Rn, \sigma: \Rn\to \R^{n\times m}, j: \Rn\times \Rl\to\R^n$: there exist   constants $C_\mu, C_\sigma>0$ and a positive function $C_j(\cdot)\in L^1\cap L^2(\Rl,\nu)$ such that
\begin{equation}\label{Lipschitz mu sigma}
  \begin{cases}
  |\mu(x)-\mu(y)|\le C_\mu|x-y|,\\
  \|\sigma(x)-\sigma(y)\|\le C_\sigma|x-y|,\\
  |j(x, z)-j(y,z)|\le C_j(z)|x-y|,
  \end{cases}
 \quad\forall x, y\in\Rn, z\in\Rl.
\end{equation}
Assume also that
\begin{equation} j(x,\cdot)\in L^1(\Rl; \nu) \quad\mbox{for every }x\in\Rn.\label{jcond}\end{equation}

\item [\textbf{(A2)}] Ellipticity: There exists a constant $\lambda>0$ such that
\begin{equation}\label{ellipticity}
a_{ij}(x)\xi_i\xi_j\ge \lambda |\xi|^2, \quad \forall x,\xi\in\Rn,
\end{equation}
where the matrix $A=\left(a_{ij}\right)_{n\times n}=\frac 1 2 \sigma(x)\sigma(x)^\textsf{T}$.

  \item [\textbf{(A3)}] Lipschitz condition on the running cost $f\ge0$: there exists a constant $C_f>0$ such that
\begin{equation}\label{Lipschitz f}
|f(x)-f(y)|\le C_f|x-y|, \quad\forall x, y\in\Rn.
\end{equation}

  \item [\textbf{(A4)}] Conditions on the transaction cost function $B:\Rn\to
  \R$:
      \begin{align}
\begin{cases}
    \DS\inf_{\xi\in\Rn}B(\xi)=K>0,\\
    B\in C(\Rn\backslash\{0\}),   \\
      |B(\xi)|\toinf, \mbox{ as } |\xi|\toinf, \mbox{ and}\\
    B(\xi_1)+B(\xi_2)\ge B(\xi_1+\xi_2)+K,\quad \forall \xi_1, \xi_2\in\Rn.
    \end{cases}\label{conditions on B}
\end{align}

\item [\textbf{(A5)}]  $r>2C_\mu+C_\sigma^2+\int_{\Rl}C_j^2(z)\nu(dz)$.
\end{itemize}

Assumption \textbf{(A1)} ensures the existence and uniqueness of solutions to \eqref{Eq. dynamics} (Cf. Theorem 9.1, Chapter VI, \cite{IkedaWatanabe89}). 
The condition \eqref{jcond} seems essential to our approach, in particular
to establish the continuity property of the operator $I$ in Lemma \ref{lem
continuous Iu}.
 Readers are referred to \cite{GS79} or \cite{OS04} for a more detailed
discussion of L\'evy processes and jump diffusions.

In view of Assumption ({\bf A1}) the following definitions for operators $L,I$
make sense.
\begin{equation}L\varphi(x)=- \tr \left[A\cdot D^2\varphi(x)\right]- \bar\mu(x)\cdot D\varphi(x)+r \vf(x),\label{L-def}
\end{equation}
where $\bar{\mu}=\mu-\int_{\Rl}j(x,z)\nu(dz)$ and, for Lipschitz continuous
functions $\varphi$,
\begin{equation}
I\varphi(x)=\int_{\Rl} \left[\varphi(x+j(x, z))-\vf(x)\right]\nu(dz)\label{I-def}
\end{equation}
We also adopt the following  standard notations for function
spaces:
\begin{align*}
  \UC&= \text{space of all uniformly continuous functions on $\Rn$},\\
  W^{k, p}(U)&= \text{space of all $L^p$ functions with $\beta$-th weak partial}\\
  &\quad\text{ derivatives  belonging to $L^p$}, \forall |\beta|\le k, \\
  W^{k,p}_0(U) &= \text{the closure, in $W^{k,p}$-norm, of smooth functions with compact support in $U$,}\\
  W^{k,p}_{\loc}(U)&=\{f\in W^{k,p}(U'), \forall \text{ compact } U'\subset U\},\\
  C^{k,\alpha}(D)&=\left\{f\in C^k(D): \sup_{x, y\in D\atop x\ne y}\left\{\frac{|D^\beta f(x)-D^\beta f(y)|}{|x-y|^\alpha}\right\}<\infty, \forall |\beta|\le k\right\}, D\text{ compact.}
\end{align*}

\section{Preliminary Results}
We first establish some preliminary results under the assumptions \textbf{(A1)-(A5)}.

\begin{lemma}
The value function $u(\cdot)$ defined by \eqref{def of u (value function)} is Lipschitz.
\end{lemma}
\begin{proof}
Given an admissible control $V$, and two initial states $x_1, x_2$, denote by $X^i(t)$ the solution of \eqref{Eq. dynamics}. Apply It\^o formula (for jump diffusions) (Theorem 5.1, Chapter II, \cite {IkedaWatanabe89}) to $Y(t)=|Z(t)|^2$, where $Z(t)=X^1(t)-X^2(t)$,
\begin{align*}
dY(t)&=2Z(t)\cdot [(\mu(X^1(t))-\mu(X^2(t)))dt+(\sigma(X^1(t))-\sigma(X^2(t)))dW]\\
&\quad+ (\sigma(X^1(t))-\sigma(X^2(t)))(\sigma(X^1(t))-\sigma(X^2(t)))^\textsf{T}dt\\
&\quad-\int_{\Rl} 2(j(X^1(t),z)-j(X^2(t),z))Z(t^-)\nu(dz)dt\\
&\quad+\int_{\Rl} [(Z(t^-)+j(X^1(t),z)-j(X^2(t),z))^2-|Z(t^-)|^2]N(dt,dz)
\end{align*}
Integrating from 0 to $t$, taking the expectation and then using Assumption \textbf{(A1)}, we obtain
\begin{align*}
\E Y(t)-(x_1-x_2)^2&\le \left(2C_\mu+C_\sigma^2+\int_{\Rl}C_j^2(z)\nu(dz)\right)\int_0^t\E Y(s)ds,
\end{align*}
which implies that $\E |X^1(t)-X^2(t)|\le e^{Ct}|x_1-x_2|$ by Gronwall's inequality, where $C=2C_\mu+C_\sigma^2+\int_{\Rl}C_j^2(z)\nu(dz)$. Hence $J_{x_1}[V]-J_{x_2}[V]\le C_u|x_1-x_2|$ by Assumptions \textbf{(A3)} and \textbf{(A5)}, where
\[C_u:=\frac{C_f}{r-[2C_\mu+C_\sigma^2+\int_{\Rl}C_j^2(z)\nu(dz)]}>0.\]

By the arbitrariness of $V$,
  \[u(x_1)\le J_{x_1}[V]\le J_{x_2}[V]+C_u|x_1-x_2|\Rightarrow u(x_1)\le u(x_2)+C_u|x_1-x_2|.\]
  Exchanging the roles of $x_1, x_2$ we get the desired result.
\end{proof}

\begin{lemma}\label{lem continuous Iu}
$Iu\in C(\Rn).$
\end{lemma}
\begin{proof}
Given $x\in\Rn$, $u(y+j(y,z))-u(y)\to u(x+j(x,z))-u(x)$, as $y\to x$, for any $z\in\Rl$. Observe that if  $|y-x|<1$,
\[|u(y+j(y,z))-u(y)|\le C_u|j(y,z)|\le C_u(|j(x,z)|+C_j(z)|x-y|)\le C_u(|j(x,z)|+C_j(z)),\]
where $C_u$ is the Lipschitz constant of $u$. Since $j(x,\cdot)$ and $C_j(\cdot)$ are both $\nu$-integrable, the dominated convergence theorem yields the desired result.
\end{proof}

For reference, we recall here Lemmas \ref{lem properties of M}-\ref{lem x+xi in C} which were proved in \cite{GW09}.
\begin{lemma}[Properties of $\M$]\label{lem properties of M}\mbox{}
\begin{enumerate}
  \item $\M$ is concave: for any $\varphi_1, \varphi_2\in C(\Rn)$ and $0\le s\le 1$,
      \[\M(s \varphi_1+(1-s)\varphi_2)\ge s \M\varphi_1+(1-s) \M\varphi_2.\]
  \item $\M$ is increasing: for any $\varphi_1\le\varphi_2$ everywhere,
      \[\M\varphi_1\le \M\varphi_2.\]
  \item $\M$ maps $\UC$ into $\UC$ and maps a Lipschitz function to a Lipschitz function. In particular, $\M u(\cdot)$ is Lipschitz continuous.
\end{enumerate}
\end{lemma}

\begin{lemma}
  $u$ and $\M u$ defined as above satisfy $u(x)\le \M u(x)$ for all $x\in\Rn$. \label{lem u<=Mu}
\end{lemma}

We define the \emph{continuation region} $\C$ and the
\emph{action region} $\A$ as follows,
\begin{align}
  \C&:=\{x\in\Rn: u(x)<\M u(x)\},\label{def of C (continuation region)}\\
  \A&:=\{x\in\Rn: u(x)=\M u(x)\}.\label{def of A (action region)}
\end{align}
Then, $\C$ is open, and we have

\begin{lemma}\label{lem x+xi in C}
  Suppose $x\in \A$, then

\noindent(i) The set
  \[\Xi(x):=\{\xi\in\Rn: \M u(x)=u(x+\xi)+B(\xi)\}\]
  is nonempty, i.e., the infimum is in fact a minimum.

\noindent(ii) Moreover, for any  \,$\xi(x)\in\Xi(x)$, we have
  \[u(x+\xi(x))\le \M u(x+\xi(x))-K,\]
  in particular,
  \[x+\xi(x)\in \C.\]
\end{lemma}
\section{Viscosity Solutions}
There are  different ways to define viscosity solutions. Let us begin with the most common one.
\begin{definition}\label{def 1 of viscosity}
A function  $u(\cdot)\in \UC$ is called a viscosity subsolution (supersolution, resp.) of \eqref{HJB} if whenever $\varphi\in C^2(\Rn), u-\varphi$ has a \textbf{global} maximum (minimum, resp.) at $x_0$ and $u(x_0)=\varphi(x_0)$, we have
      \begin{align}
        \max\{\Ell \varphi(x_0)-f(x_0), \varphi(x_0)-\M \varphi(x_0)\}\le 0\quad (\ge 0\text{ resp.});\label{viscosity solution 1}
      \end{align}
and $u$ is called a viscosity solution of \eqref{HJB} if it is both a subsolution and a supersolution.
\end{definition}

Besides this standard definition of viscosity solutions,  there are at least another two different (but equivalent) ones. The second way is to use semijets in stead of test functions. See, for instance, \cite{CIL92} and \cite{Seydel09}, for more details. For the purpose of proving our regularity results in Section~\ref{sec regularity}, we give a third definition below.
The idea is that we impose ``local'' conditions (rather than global conditions as in Definition~\ref{def 1 of viscosity}) on the test functions, and in the equation we only replace $u$ by the test function $\vf$ in the ``local'' terms while still keep $u$ in the ``nonlocal'' terms.
The same definition (in different notation) and the proof of equivalence can be found in \cite{Seydel08}.
See also \cite{Soner86a, BI08} and   \cite{GW09} for a similar treatment.

\begin{definition}\label{def 2 of viscosity}
A function  $u(\cdot)\in \UC$ is called a viscosity subsolution (supersolution, resp.) of \eqref{HJB} if whenever $\varphi\in C^2(\Rn), u-\varphi$ has a \textbf{local} maximum (minimum, resp.) at $x_0$ and $u(x_0)=\varphi(x_0)$, we have
\begin{align} \label{viscosity solution 2}
\max\{L \vf(x_0)-f(x_0)-I u(x_0), u(x_0)-\M u(x_0)\}\le 0\quad (\ge 0\text{ resp.}).
\end{align}
$u$ is called a viscosity solution of \eqref{HJB} if it is both a subsolution and a supersolution.
\end{definition}

\begin{theorem}\label{thm equivalence of definitions}
The above two definitions of viscosity solutions are equivalent.
\end{theorem}
\begin{proof}
See \cite[Proposition 5.4]{Seydel08}.
\end{proof}

We now have the following basic result.
\begin{theorem}[\cite{OS04,Seydel09}]
\label{thm viscosity}
The value function $u(\cdot)$ defined by \eqref{def of u (value function)} is a viscosity solution of \eqref{HJB}.
\end{theorem}

This theorem was proved in \cite[Theorem 9.8]{OS04} as well as \cite[Theorem 4.2]{Seydel09}\footnote{In \cite{Seydel09} the result was proved using an
in principle smaller class of controls, the so-called `Markov controls'. However,
this restriction is unnecessary, as can be seen from the proofs of the analogous
results in \cite{TY93} or \cite{Ishikawa04}, or \cite{YongZhou99}.}
in the sense of our Definition \ref{def 1 of viscosity}. But when we prove the regularity result below, we found it more convenient to use Definition \ref{def 2 of viscosity}. More precisely, by Theorem~\ref{thm equivalence of definitions} and Theorem~\ref{thm viscosity}, we can ``identify'' our value function $u(\cdot)$ with that of an impulse control problem of diffusion processes without jumps.
\begin{corollary}
\label{cor relate to diffusions}
The value function $u(\cdot)$ is a viscosity solution of
\begin{align}\label{local equation}
\max\{L u(x)-\tilde f(x), u(x)-\M u(x)\}=0\quad \text{in }\Rn,
\end{align}
where $\tilde f(x)=f(x)+I u(x).$
\end{corollary}

\section{Regularity of Value Function}\label{sec regularity}
In this section we study the smoothness of the value function $u$, starting with the special case of a finite L\'evy measure.

\subsection{Special Case: $\nu(\Rl)<\infty$}
Let us first consider the special case in which the L\'evy measure is finite, or equivalently, the jump diffusion $X(\cdot)$ is driven by a compound Poisson process.
Then the operator $I$ enjoys the following nice property.

\begin{lemma}\label{lem Lipschitz Iu}
Suppose $\nu(\Rl)<\infty$, then the operator $I$ maps a Lipschitz function to a Lipschitz function.
\end{lemma}
\begin{proof}
Suppose $\vf(x)$ is Lipschitz with $|\vf(x)-\vf(y)|\le C_\vf|x-y|$ for any $x, y\in\Rn$, then
\begin{align*}
|I \vf(x)-I \vf(y)|&\le \int_{\Rl}|\vf(x+j(x, z))-\vf(y+j(y, z))|\nu(dz)+\int_{\Rl}|\vf(x)-\vf(y)|\nu(dz)\\
&\le C_\vf\int_{\Rl}\left[2|x-y|+|j(x,z)-j(y,z)|\right]\nu(dz)\\
&\le C_\vf\left(2\nu(\Rl)+\int_{\Rl}C_j(z)\nu(dz)\right)|x-y|.
\end{align*}
So $I\vf$ is Lipschitz.
\end{proof}

Corollary ~\ref{cor relate to diffusions} and Lemma \ref{lem Lipschitz Iu} together imply the regularity of $u$ in the continuation region.
\begin{lemma}
[\boldmath$C^{2,\alpha}$-Regularity in $\C$]\label{lem regularity in C}
Assume that $\sigma\in C^1(\Rn)$ and $\nu(\Rl)<\infty$,
then for any compact set $D\subset \C$, the value function $u(\cdot)$ is in the H\"older space $C^{2,\alpha}(D)$ for any $\alpha\in (0,1)$, and it is a classical solution of
\[\Ell u- f(x)=0\text{ in }\C.\]
\end{lemma}
\begin{proof}
Note that $u$ is a viscosity solution of \eqref{local equation} by Corollary~\ref{cor relate to diffusions}, and hence a viscosity solution of $Lu-\tilde f=0$ in $\C$. On the other hand,
$\tilde f\in C^\alpha$ for any $\alpha< 1$ by Lemma~\ref{lem Lipschitz Iu}. Classical Schauder estimates imply the desired results. (See the proof of Lemma~\ref{lem W2p regularity in C} below for a similar argument.)
\end{proof}

Finally, an argument as in \cite[\S 4]{GW09} applies and yields the following
\begin{theorem}
  [\boldmath$W^{2, p}_{\loc}$-Regularity] Assume that $\nu(\Rl)<\infty$ and
   \begin{align}
     &\sigma\in C^{1,1}(D) \text{ for any compact set }D\subset \Rn.\label{more condition on sigma}
   \end{align}
   Then for any bounded open set $\mathcal{O}\subset \Rn$ and $p<\infty$, we have $u\in W^{2, p}(\mathcal{O})$.
\end{theorem}

\subsection{More General Case: $j(x,\cdot)\in L^1(\nu)$}
Next, we would like to remove the strong assumption that the L\'evy measure $\nu$ is finite, and assume only our standing assumptions \textbf{(A1)-(A5)}.
Again, we first consider the regularity of $u$ in the continuation region $\C$, in which the linear elliptic PDE is satisfied. The difficulty is that we do not know $Iu$ is Lipschitz or even H\"older continuous, but only continuous, by Lemma~\ref{lem continuous Iu}.

We cannot apply Schauder estimates at this stage, but the $L^p$ estimates give the following
\begin{lemma}
[\boldmath$W^{2,p}_{\loc}$-Regularity in $\C$]\label{lem W2p regularity in C}
Assume that $\sigma\in C^1(\Rn)$, then for any compact set $D\subset \C$, the value function $u(\cdot)$ is in the Sobolev space $W^{2,p}(D)$ for any $p<\infty$, and it is a strong solution\footnote{A strong solution is a twice weakly differentiable function in the bounded domain that satisfies the equation almost everywhere.} of
\[\Ell u- f(x)=0\text{ in }\C.\]
\end{lemma}

\begin{proof}
Denote by $\tilde f=f+Iu$, which is continuous by Lemma~\ref{lem continuous Iu}. Consider in any open ball $B\subset\C$ the following Dirichlet problem
\begin{align}\label{eqn dirichlet in ball}
  \begin{cases}
    Lw=\tilde f, &\text{in }B,\\
    w=u, &\text{on }\partial B.
  \end{cases}
\end{align}
Classical $L^p$ theory (Cf. \cite[Corollary 9.18]{GT98}) asserts that the Dirichlet problem \eqref{eqn dirichlet in ball} has a unique strong solution $w\in W_{\loc}^{2,p}(B)\cap C(\bar B)$ for any $p<\infty$, since $\tilde f\in C(\bar B)$ and the boundary data $u\in C(\partial B)$. Because $\sigma\in C^1(B), \bar\mu\in C^{0,1}(B)$ and $\tilde f\in C(B)$, this solution $w$ is in fact also a viscosity solution of \eqref{eqn dirichlet in ball} by \cite[Theorem 2]{Ishii95}.

On the other hand, $u$ is also a viscosity solution of \eqref{eqn dirichlet in ball} by Corollary~\ref{cor relate to diffusions}. Therefore, $w=u$ in $\bar B$ by classical uniqueness results of viscosity solutions to a linear elliptic PDE in a bounded domain(Cf. \cite[Theorem 3.3]{CIL92}). Hence $u\in W_{\loc}^{2,p}(B)\cap C(\bar B)$.

Finally, any compact set $D\subset \C$ can be covered by finitely many balls $\{B_r(x_k)\}_{k=1}^N$ of radius $r<\frac{1}{2}\dist(D,\partial \C)$. Let $B=B_{2r}(x_k)\subset \C$ in the above argument, then $u$ is in $W^{2,p}(\bar B_r(x_k))$ for all $k$ and also in $W^{2,p}(D)$.
\end{proof}

With more regularity of $u$ in the continuation region $\C$, we can use the ``bootstrap argument'' to obtain further regularity of $Iu$ (and hence $u$) in $\C$.

\begin{theorem}[\boldmath$C^{2,\alpha}$-Regularity in $\C$]\label{thm more regularity in C}
Assume that $\sigma\in C^1(\Rn)$,
then for any compact set $D\subset \C$, the value function $u(\cdot)$ is in the H\"older space $C^{2,\alpha}(D)$ for any $\alpha\in (0,1)$, and it is a classical solution of
\[\Ell u- f(x)=0\text{ in }\C.\]
\end{theorem}
\begin{proof}
The key step in the proof is to show $Iu\in C^\alpha(D)$ for any compact $D\subset \C$.

Take a compact set $D'$ such that $D\subset D'\subset \C$ and $\delta:=\dist(D, \partial D')<1$. Then by Lemma~\ref{lem regularity in C}, $u\in W^{2,p}(D')$ for any $p<\infty$. By Sobolev imbedding, $u\in C^{1, \alpha}(D')$ for all $\alpha\in (0,1)$. Define the set
\[E:=\{z\in\Rl: |j(x,z)|<\delta, \forall x\in D\}.\]
Then for $z\in E^c=\Rl\setminus E$, there is $x\in D$ such that \begin{align*}
 \delta\le |j(x, z)|\le |j(0, z)|+C_j(z)|x|\le  |j(0, z)|+C_DC_j(z),
\end{align*}
where $C_D=\max\{|x|:x\in D\}$ is a constant. So $|j(0,z)|\ge \delta/2$ or $C_j(z)\ge \delta/(2C_D)$ and
\begin{align*}
  \nu(E^c)\le \frac{2}{\delta}\int_{\Rl} |j(0,z)|\nu(dz)+\frac{2C_D}{\delta}\int_{\Rl}C_j(z)\nu(dz)<\infty.
\end{align*}
For any $x_1, x_2\in D$,
\begin{align}
  |Iu(x_1)-Iu(x_2)|&\le \int_{E}|[u(x_1+j(x_1, z))-u(x_1)]-[u(x_2+j(x_2,z))-u(x_2)]\nu(dz)\nonumber\\
  &\quad+\int_{E^c}|u(x_1+j(x_1,z))-u(x_2+j(x_2,z))|+|u(x_1)-u(x_2)|\nu(dz)\nonumber\\
  &\le \int_E \int_0^1 |Du(x_1+sj(x_1, z))\cdot j(x_1, z)-Du(x_2+sj(x_2,z))\cdot j(x_2,z)|ds\;\nu(dz)\nonumber\\
  &\quad+|x_1-x_2|\int_{E^c}C_u(2+C_j(z))\nu(dz).\nonumber\\
  &\le \int_E\int_0^1 |Du(x_1+sj(x_1,z))-Du(x_2+sj(x_2,z))|\cdot|j(x_1,z)|ds\;\nu(dz)\nonumber\\
  &\quad + \int_E\int_0^1 |Du(x_2+sj(x_2,z))|\cdot|j(x_1,z)-j(x_2,z)|ds\;\nu(dz)\nonumber\\
  &\quad+|x_1-x_2|\int_{E^c}C_u(2+C_j(z))\nu(dz).\label{eqn Iu(x_1)-Iu(x_2)}
 \end{align}
Note that $x_1+sj(x_1,z), x_2+sj(x_2,z)\in D'$ for all $0\le s\le 1,  z\in E$ and that $Du\in C^\alpha(D')$. Thus the first integral in \eqref{eqn Iu(x_1)-Iu(x_2)} can be estimated by
\begin{align*}
  &\quad\|Du\|_{C^\alpha(D')}\int_E\left(\int_0^1|x_1-x_2+s(j(x_1,z)-j(x_2,z))|^\alpha ds\right)|j(x_1,z)|\nu(dz)\\
  &\le  \|Du\|_{C^\alpha(D')} \int_E |x_1-x_2|^\alpha(1+C_j(z)^\alpha) |j(x_1,z)|\nu(dz)\\
  &\le  C_1|x_1-x_2|^\alpha,
\end{align*}
for some constant $C_1>0$ independent of $x_1, x_2$, because by H\"older's inequality,
\begin{align*}
  &\quad \int_E (1+C_j(z)^\alpha) |j(x,z)|\nu(dz)\\
  &\le \int_{\Rl}(1+C_j(z)^\alpha) |j(x,z)|\mathbf{1}_{\{z:|j(x,z)|<1\}}\nu(dz)\\
  &\le \int_{\Rl}|j(x,z)|\nu(dz)+ \left(\int_{\Rl} C_j(z)^2\nu(dz)\right)^{\frac{\alpha}{2}}\left(\int_{\Rl} |j(x,z)|^{\frac{2}{2-\alpha}}\mathbf{1}_{\{z:|j(x,z)|<1\}} \nu(dz)\right)^{\frac{2-\alpha}{2}}\\
  &\le \int_{\Rl}|j(x,z)|\nu(dz)+\left(\int_{\Rl} C_j(z)^2\nu(dz)\right)^{\frac{\alpha}{2}}\left(\int_{\Rl}|j(x,z)| \nu(dz)\right)^{\frac{2-\alpha}{2}},
\end{align*}
which is a continuous function in $x$ and has a maximum on $D$ independent of $x_1, x_2$.

The second term in \eqref{eqn Iu(x_1)-Iu(x_2)} can be majored by $\|Du\|_{L^\infty(D)}\int_{\Rl} C_j(z)\nu(dz) |x_1-x_2|=:C_2|x_1-x_2|$ and the third term is majored by $C_u[2\nu(E^c)+\int_{\Rl}C_j(z)\nu(dz)]|x_1-x_2|=:C_3|x_1-x_2|$. Thus, putting all three terms in \eqref{eqn Iu(x_1)-Iu(x_2)} together,
\begin{align*}
  |I(x_1)-I(x_2)|\le C_1|x_1-x_2|^\alpha +(C_2+C_3)|x_1-x_2|\le C|x_1-x_2|^\alpha,
\end{align*}
where $C_1, C_2, C_3$ and $C=C_1+(C_2+C_3)(\diam D)^{1-\alpha}$ are constants independent of $x_1, x_2\in D$. This proves that $Iu\in C^\alpha(D)$.

Finally, we can repeat a similar argument in the proof of Lemma~\ref{lem regularity in C}. This time we know $\tilde f\in C^\alpha(D)$, thus the solution of \eqref{eqn dirichlet in ball} is in fact in $C^{2,\alpha}(D)$ by Schauder estimates (\cite[Theorem 6.13]{GT98}). Thus, $u\in C^{2,\alpha}(D)$ for compact  $D\subset \C$, and $u$ is a classical solution of $\Ell u-f=0$ in $\C$.
\end{proof}

Parallel to \cite{GW09}, once we have $C^{2,\alpha}$ regularity, we are able to obtain $W^{2,p}(D)$ regularity for any compact set $D\subset\Rn$.

\begin{theorem}[\boldmath$W^{2, p}_{\loc}$-Regularity] \label{thm regularity of u}
Assume that
   \begin{align}
     \sigma\in C^{1,1}(D) \text{ for any compact set }D\subset \Rn.\label{same condition on sigma}
   \end{align}
Then for any bounded open set $\mathcal{O}\subset \Rn$ and $p<\infty$, we have $u\in W^{2, p}(\mathcal{O})$.
In particular, $u\in C^1(\Rn)$ by Sobolev embedding.
\end{theorem}

The proof is given in the next section. It follows similar lines to that
of \cite[Theorem 4.2]{GW09}, involving regularity of an associated optimal
stopping problem.

\section{Proof of Theorem \ref{thm regularity of u}}

To study the regularity of the value function for the impulse control problem, we need to investigate the related optimal stopping problem.
More precisely, we shall obtain the regularity of solutions for the HJB equation associated with this optimal stopping problem.
\begin{theorem}\label{thm:optimal stopping regularity}
Suppose $\co$ is a bounded open set  in $\Rn$ with smooth boundary. Assume
\begin{align}
  &a_{ij}\in C^{1,1}(\ob), \quad \mu_i\in C^{0,1}(\ob),  \quad r>0, \quad f \in C(\ob),\label{eqn:coeff conditions}\\
  &a_{ij}\xi_i\xi_j\ge c|\xi|^2, \forall x, \xi\in\ob,\text{ for some }c>0,\label{eqn:ellipticity1}\\
  &g\in C(\ob), \quad g\ge 0\text{ on }\partial \co.\label{eqn:g nonnegative on boundary}
\end{align}
Assume also that there exist a sequence of functions $\{g^\ve\}_{\ve>0}$ and a  constant $M>0$  satisfying
\begin{align}\label{eqn:condition on g}
\begin{cases}
    g^\ve\in C^2(\co)\cap C(\ob), \quad Lg^\ve\ge -M\text{ in }\co,\\
    g^\ve \to g \text{ uniformly in }\ob.
  \end{cases}
\end{align}

   If $v\in C(\ob)$ is a viscosity solution of
\begin{equation}\label{eqn:optimal stopping Dirichlet}
\begin{cases}
  \max\{Lv-f, v-g\}=0&\text{  in } \co, \\
   v=0&\text{ on }\partial \co,
  \end{cases}
  \end{equation}
then $v\in W^{2, p}(\co)$.
\end{theorem}

\begin{remark}
Note that Assumption \eqref{eqn:condition on g} is trivially satisfied if $g\in C^2(\ob)$. However,  later we will apply this theorem to $g=\M u$, which is not necessarily in $C^2(\ob)$. In applications, $g^\ve$ can be taken as the usual mollification of $g$, or its slight modification (which may only be in $C^2(\co)\cap C(\ob)$ but not in $C^2(\ob)$, as in Corollary~\ref{cor:optimal stopping regularity} below).
\end{remark}

As a corollary of Theorem~\ref{thm:optimal stopping regularity}, we obtain local $W^{2,p}$ ($n<p<\infty$) regularity of continuous viscosity solutions of
\begin{equation}
\max\{ L v-f, v-g\}=0 \mbox{ in }\Rn. \label{eqn:HJB optimal stopping}
\end{equation}

\begin{corollary}\label{cor:optimal stopping regularity}
Assume that $f\in C(\Rn)$,
  $a_{ij}\in C^{1,1}_{\loc}(\Rn)$, and $\mu,\sigma$ and $g$ are Lipschitz in $\Rn$.
Assume also that for any bounded open set $\co\subset \Rn$ with smooth boundary, there are constants (maybe depending on $\co$) $c>0$ and $M$  such that \eqref{eqn:ellipticity1} and \eqref{eqn:condition on g} are satisfied.

If $v\in C(\Rn)$ is a viscosity solution of \eqref{eqn:HJB optimal stopping}, then $v\in W^{2, p}(\co)$ for any $\co\subset \Rn$ with smooth boundary and any $1\le p<\infty$, and hence also in $C^1(\Rn)$.
\end{corollary}

We defer the proofs of Theorem \ref{thm:optimal stopping regularity}
and its corollary to the appendix, and focus now on proving our main theorem using the above corollary.

\begin{proof}[Proof of Theorem \ref{thm regularity of u}]
Given any bounded open set $\mathcal{O}$ with smooth boundary, we denote by $\C'$ ($\A'$, resp.) the restriction of the continuation (action, resp.) region within $\mathcal{O}$.
Then there exists an open ball $\co'\supset \co$ such that for any $x\in \co$, $u(x+\xi)+B(\xi)\le \M u(x)+1$ implies $x+\xi\in \co'.$
Because in this case,
\[B(\xi)\le \M u(x)-u(x+\xi)+1\le \M u(x)+1\le \sup_{\ob} \M u+1<\infty.\]
But $B(\xi)\toinf$ as $|\xi| \toinf$, which implies that all such $\xi$ are bounded uniformly.

Now we define the set
\begin{align}
  D:=\left\{y\in\mathcal{O}': u(y)<\M u(y)-\frac K 2\right\}.
\end{align}
Clearly, $\overline{D}$ is compact and $\overline{D} \subset \mathcal{C}$. From Lemma \ref{lem regularity in C}, \[u\in C^{2,\alpha}(\overline D).\]

For any $x\in\co$, take a minimizing sequence $\{\xi_k\}$ such that $u(x+\xi_k)+B(\xi_k)\to \M u(x)$. Then $\{\xi_k\}\subset \co'$. Extract a convergent subsequence (still denoted by $\{\xi_k\}$) converging to $\xi^*$.  Because $B(\xi)+B(\xi')\ge K+B(\xi+\xi')$,
   \begin{align*}
     \M u(x)
     &=\inf_{\eta\in\Rn} \{u(x+\xi_k+\eta)+B(\xi_k+\eta)\}\\
     &\le \inf_{\eta\in\Rn}\{u(x+\xi_k+\eta)+B(\eta)\}+B(\xi_k)-K\\
     &=\M u(x+\xi_k)+B(\xi_k)-K\\
     &=\M u(x+\xi_k)-u(x+\xi_k)+[u(x+\xi_k)+B(\xi_k)]-K.
   \end{align*}
Passing to the limit $k\toinf$, we obtain
\[u(x+\xi^*)-\M u(x+\xi^*)\le -K.\]
In particular, $y:=x+\xi^*\in D$.

On the other hand, since $u-\M u$ is uniformly continuous on
$\overline {\mathcal{O}'}$, there exists  $\rho_0>0$ such that
\[|y-y'|\le \rho_0\Rightarrow |u(y')-\M u(y')-(u(y)-\M u(y))|\le \frac K 4.\]
Hence,   for all $\rho\in (0, \rho_0], \lambda\in[-1,1]$ and unit vector $\chi\in\Rn$,
\begin{align*}
y=x+\xi^*\in D,\quad y'=y+ \lambda\rho\chi\in D,
\end{align*}
because $u(y')-\M u(y')\le u(y)-\M u(y)+\frac K 4<-\frac K 2.$

Since $\M u(x\pm \rho\chi)\le u(x\pm \rho\chi+\xi_k)+B(\xi_k)$ for all $k$,
\begin{align*}
&\M u(x+\rho \chi)+\M u(x-\rho \chi)-2\M u(x)\\
\le & u(x+\rho\chi+\xi_k)+u(x-\rho\chi+\xi_k)+2B(\xi_k)-2\M u (x)\\
\to & u(y+\rho\chi)+u(y-\rho\chi)-2u(y), \quad k\toinf,
\end{align*}
and hence the second order difference quotient at $x$
\begin{align*}
  &\frac 1 {\rho^2}[\M u(x+\rho\chi)+\M u(x-\rho\chi)-2\M u(x)]\\
  \le& \frac 1 {\rho^2} [u(y+\rho\chi)+u(y-\rho\chi)-2u(y)]\\
  =&\frac 1 {|\rho|}\int_0^1 \left[(Du (y+\lambda\rho\chi)-Du(y-\lambda\rho\chi)\right]\cdot\chi d\lambda\\
  \le & C_D,
\end{align*}
where $\DS C_D=\sup_{x\in \overline D} |D^2u(x)|\le \|u\|_{C^{2,\alpha}(\overline D)}$.

For simplicity, denote by $g=\M u$ and $g^\ve$ its mollification.
For any $x_0\in \co$,  suppose $B_\theta(x_0)\subset \co$, then for any $\ve\in (0, \frac \theta 2),  \rho\in (0, \rho_0\wedge\frac{ \theta }{2})$, and a unit vector $\chi\in\Rn$,
\begin{align*}
  &\frac1{\rho^2}\left[g^\ve(x_0+\rho\chi) +g^\ve(x_0-\rho\chi) -2 g^\ve (x_0)\right]\\
  =&\frac1{\rho^2}\int_{B_\varepsilon(0)}[g(x_0-z+\rho\chi)+ g(x_0 -z - \rho \chi)-2g(x_0-z)]\eta^\varepsilon(z)\,dz\\
  \le& C_D\int_{B_\varepsilon(0)} \eta^\varepsilon(z)\,dz=C_D.
\end{align*}
Sending $\rho\to0$ we get
\[\chi^\textsf{\textup T} D^2g^\ve(x_0)\chi\le C_D.\]
Hence,
\begin{align*}
  \tr(\sigma(x_0)\sigma(x_0)^\textsf{\textup T}D^2g^\ve(x_0))&=  \tr(\sigma^\textsf{\textup T}(x_0)D^2g^\ve(x_0)\sigma(x_0))\\
  &= \sum_k \sigma_k^\textsf{\textup T} D^2g^\ve \sigma_k\\
  &\le C_D\sum_{i,j}|\sigma_{ij}(x_0)|^2\\
  &\le C,
  \end{align*}
  where $\sigma_k$ is the $k$-th column of the matrix $\sigma$, $\sigma_{ij}$ is the $(i,j)$-th element of $\sigma,$ and the last inequality is due to continuity of $\sigma$.

Note that $|g^\ve(x_0)|+|Dg^\ve(x_0)|\le \|g\|_{W^{1,\infty}({\mathcal{O}})}$ and $\mu(x)$ bounded,  we deduce
\[L g^\ve(x_0)=-\frac 1 2\tr\left(\sigma(x_0)\sigma(x_0)^\textsf{\textup T} D^2g^\ve(x_0)\right)-\mu(x_0) \cdot Dg^\ve(x_0)+rg^\ve(x_0)\ge -M,\]
where the constant $M$ is independent of $x_0.$

Finally, recall that $u$ is a viscosity of \eqref{local equation} by Corollary~\ref{cor relate to diffusions}.
We can apply Corollary~\ref{cor:optimal stopping regularity} with $f$ replaced by $\tilde f=f+Iu\in C(\Rn)$ and $g=\M u$ and conclude that for any $1\le p<\infty$,
\[u\in W^{2,p}(\co).\]
\end{proof}

\nocite{GM02}
\bibliographystyle{abbrv}

\appendix
\section{Proof of Theorem \ref{thm:optimal stopping regularity}}
A standard technique to get regularity is to consider a sequence of penalized problems.
For this, let $\be$ denote a sequence of smooth functions satisfying
\begin{align}\label{eqn:penalization functions}
 \begin{cases}
    \be(t)\to \infty,\text{ as }\ve\to 0, t>0; \\
    \be(t)\to 0, \text{ as } \ve \to 0, t\le 0;\\
    0<\be'(t)<\omega(\ve)^{-1}, \forall t;\\
     \be(0)=0, \be\ge -1.
  \end{cases}
\end{align}
Here $\omega(\cdot)$ is the modulus of continuity for the convergence $g^\ve\to g$, i.e.,
\[\omega(\delta):=\sup_{\ve\le \delta} \|g^\ve-g\|_{C(\ob)}.\]
Thus, $\omega(\ve)\to 0$ as $\ve \to 0$.
For a given $\ve>0$, the graph of $\be$ is shown in Figure~\ref{fig:penalization functions}.

\begin{figure}
\centering 
  \setlength{\unitlength}{1bp}%
  \begin{picture}(302.80, 217.39)(0,0)
  \put(0,0){\includegraphics{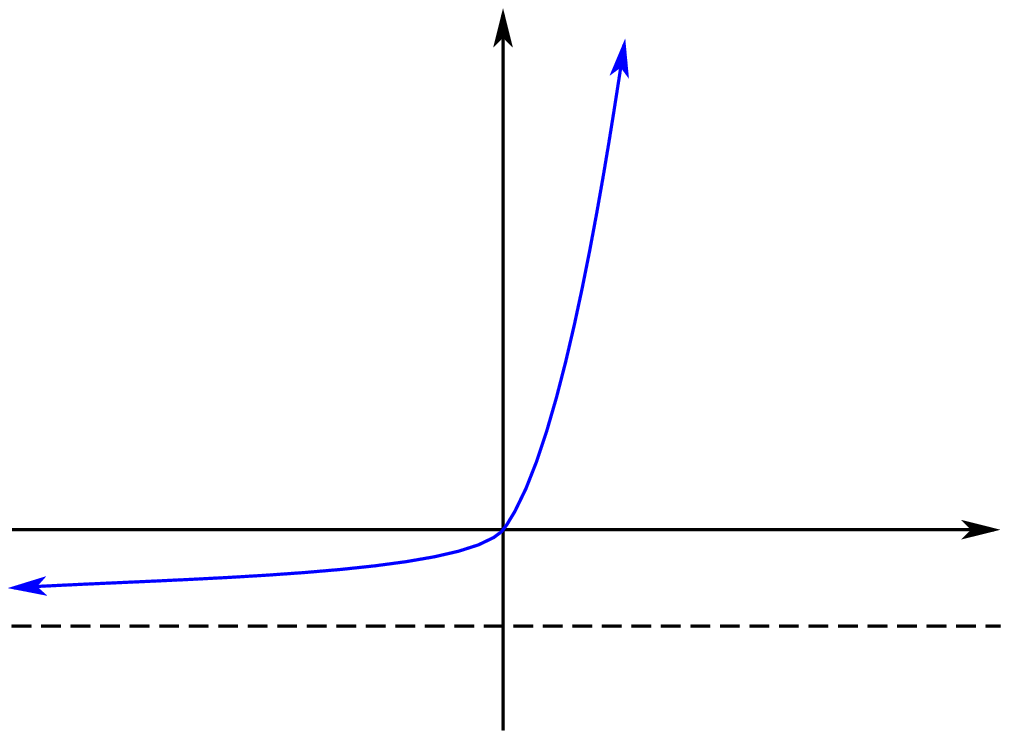}}
  \put(279.01,46.52){\fontsize{14.23}{17.07}\selectfont $t$}
  \put(179.67,160.91){\fontsize{14.23}{17.07}\selectfont $\be$}
  \put(153.78,26.65){\fontsize{8.54}{10.24}\selectfont $-1$}
  \put(150.17,54.34){\fontsize{8.54}{10.24}\selectfont $0$}
  \end{picture}%
\caption{\label{fig:penalization functions}%
 Penalizing Functions}
\end{figure}

We approximate the Dirichlet problem \eqref{eqn:optimal stopping Dirichlet} with the following penalizing problems:
  \begin{equation}\label{eqn:penalizing}
   \begin{cases}
  L  \veps +\be(\veps- g^\ve)=f &\text{ in } \co,\\
   \veps=0 &\text{ on }\partial\co.
  \end{cases}
\end{equation}
By a standard fixed point argument,  \eqref{eqn:penalizing} has a unique solution in $W^{2, p}(\co)\cap W^{1, p}_0(\co)$, for any $1\le p<\infty$. Moreover, we have the following estimates:
\begin{lemma}\label{lem:penalizing}
Under the same assumptions as Theorem~\ref{thm:optimal stopping regularity}, there exists a constant $C$ independent of $\ve$, such that
\[\|\veps\|_{W^{2,p}(\co)}\le C.\]
\end{lemma}
\begin{proof}
The goal is to show that
\begin{equation}\label{eqn:bounded be}
  \|\be(\veps-g^\ve)\|_{L^\infty(\co)} \le C.
\end{equation}
Clearly, to get \eqref{eqn:bounded be}, it suffices to show
\begin{equation}\label{eqn:upper bound of be}
  \be(\veps-g^\ve)  \le C,
\end{equation}
since $\be\ge -1$ by construction.

Consider the point $x_0$ at which the maximum of  $\be(\veps-g^\ve)$ occurs.

\emph{Case 1. } If $x_0\in \partial \co$, then $\veps (x_0)=0\le g(x_0)$. Thus
\begin{align*}
\veps(x_0)- g^\ve(x_0)&\le g(x_0)-g^\ve(x_0)\le \omega(\ve).
\end{align*}
Because $0<\be' <\omega(\ve)^{-1}$ and  $\be(0)=0$, we have $\be(\veps-g^\ve)\le 0$, if $\veps-g^\ve\le 0$; or $0\le \be(\veps-g^\ve)\le 1$ if $\veps-g^\ve\ge 0$. In either case,
we have $\be(\veps-g^\ve) \le 1$ at $x_0$.

\emph{Case 2. } If $x_0\in \co$, then $\veps-g^\ve$ also attains its maximum there, since $\be'>0$. By Bony's maximum principle (see \cite{Bony67} or \cite{Lions83c}),
\[\ess\lim_{x\to x_0}\sup \{-a_{ij}(v^\ve-g^\ve)_{x_ix_j}\}\ge 0.\]
Thus, either $v^\ve-g^\ve\le 0$ at $x_0$, which implies  $\be(\veps-g^\ve) \le 0$, or,
\[\ess\lim_{x\to x_0}\sup Lv^\ve\ge L g^\ve(x_0)\ge -M,\]
due to \eqref{eqn:condition on g}. By continuity,
\begin{align*}
\be(\veps-g^\ve)(x_0) = f(x_0)- \ess\lim_{x\to x_0}\sup L \veps\le f(x_0)+M\le C.
\end{align*}

Combing these two cases, we obtain \eqref{eqn:upper bound of be}.
Thus, we proved \eqref{eqn:bounded be}, and from the PDE \eqref{eqn:penalizing} we have
$\|L \veps\|_{L^\infty(\co)}\le C$ independently of $\ve$.

Finally, the Calderon-Zygmund estimate implies the desired result.
\end{proof}

Thanks to Lemma~\ref{lem:penalizing}, we can extract a subsequence of $\veps$, denoted by $\vek$, such that
\begin{align*}
   \begin{cases}
    \vek \rightharpoonup \bar v \text{ weakly in } W^{2,p}(\co),\\
    \vek \to \bar v \text{ uniformly in }\co.
  \end{cases}
\end{align*}
Due to the stability result of viscosity solutions (see, for instance,\cite[Theorem 8.3]{Crandall97}),
this limit function $\bar v$ is in fact a viscosity solution of \eqref{eqn:optimal stopping Dirichlet}.

\begin{lemma}\label{lem:limit function}
The limit function $\bar v$ is a viscosity solution of \eqref{eqn:optimal stopping Dirichlet}.
\end{lemma}
\begin{proof}
(Subsolution) Suppose $\phi(x)$ is a smooth test function and $\bar v-\phi$ has a local maximum at $x_0\in \co$ with $\bar v(x_0)=\phi(x_0)$. We want to show that
\[\max\{L \phi -f, \phi- g\}\le 0 \text{ at } x_0.\]

Without loss of generality, we may assume $x_0$ is a strict local maximum because otherwise we can replace $\phi(x)$ by $\phi(x)-|x-x_0|^4$ and prove the same result. In this case, for any open ball with radius $\delta>0$ and centered at $x_0$, denoted by $B_\delta(x_0)$, $\vek-\phi$ has a local maximum $x_k\in B_\delta (x_0)$ for $k$ sufficiently large, because $\vek$ converges to $v$ uniformly in $B_\delta(x_0)$ and $\bar v(x_0)-\phi(x_0)>\max_{\partial B_\delta(x_0)} (\bar v-\phi)$. Let $\delta$ go to zero, and we obtain
\[x_k \text{ is  a local maximum of }\vek-\phi,\quad \text{ and }\lim_{k\toinf}x_k=x_0.\]

By the maximum principle, $\ess\lim\limits_{x\to x_k}\sup \{-a_{ij}(\vek-\phi)_{x_ix_j}\}\ge 0$ at $x_k$. However, $\vek$ satisfies \eqref{eqn:penalizing}, thus,
\begin{align}
L\phi(x_k)&\le \ess\lim_{x\to x_k}\sup L \vek -r(\vek(x_k)-\phi(x_k))\nonumber\\ &=f(x_k)-\bek(\vek-\gek)(x_k)-r(\vek(x_k)-\phi(x_k)). \label{eqn:at each xk}
\end{align}
Note that $\vek-\gek\to \bar v- g$ and $\vek-\phi\to \bar v- \phi$ locally uniformly, we have
\begin{equation}\label{eqn:convergent sequence}
\begin{cases}
 \vek(x_k)-\gek(x_k)\to \bar v(x_0)-g(x_0), \\
  \vek(x_k)-\phi(x_k)\to \bar v(x_0)-\phi(x_0)=0,
\end{cases}
\text{ as } k\toinf.
 \end{equation}
Sending $k\toinf$ in \eqref{eqn:at each xk}, we obtain
$$\phi -g=\bar v -g \le 0  \text{ at }x_0.$$
(Otherwise, $L\phi\le -\infty$ at $x_0$ since $\be(t)\toinf$ as $\ve\to0$ if $t>0$). Moreover,
\[L\phi\le f \text{ at }x_0,\]
since $\be(t)\to 0$ as $\ve\to 0$ if $t\le 0$. We have proved that $\bar v$ is a viscosity subsolution.

(Supersolution) Similarly, suppose $\phi(x)$ is a smooth test function and $\bar v-\phi$ has a local strict minimum at $x_0$ with $\bar v(x_0)=\phi(x_0)$. We want to show that
\[\max\{L \phi -f, \phi- g\} \ge 0  \text{ at } x_0.\]
For the same reason as above, we can take a sequence $\{x_k\}$ so that
\[x_k \text{ is  a local minimum of }\vek-\phi,\quad x_k\to x_0,\]
and \eqref{eqn:convergent sequence} still holds.
And again by maximum principle, at $x_k$,
\begin{align}\label{eqn:at each xk 2}
L\phi\ge  f-\bek(\vek-\gek)-r(\vek-\phi).
\end{align}

If $\phi- g\ge 0$ at $x_0$, then we have the desired inequality. Otherwise, $\phi(x_0)- g(x_0)=-2\nu$ for some  $\nu>0$. So $\vek(x_k)-\gek(x_k)\le -\nu<0$ for $k$ sufficiently large by \eqref{eqn:convergent sequence}, and hence sending $k\toinf$ in \eqref{eqn:at each xk 2} yields
\[L\phi \ge f \text{ at }x_0.\]
Thus, $\bar v$ is a supersolution.

Finally, the boundary condition is satisfied since $\bar v=0$ on $\partial\co$.
\end{proof}

\begin{proof}[Proof of Theorem~\ref{thm:optimal stopping regularity}]
Because of the above Lemma~\ref{lem:limit function} and the uniqueness of viscosity solutions for \eqref{HJB}, we conclude that
$v=\bar v \in W^{2, p}(\co)$.
\end{proof}

\begin{proof}[Proof of Corollary \ref{cor:optimal stopping regularity}.]
To apply Theorem~\ref{thm:optimal stopping regularity}, we need to subtract from $v$ a function that has the same boundary value on $\partial\co$. Let  $w$ be the unique solution of the Dirichlet problem
\begin{align*}
   \begin{cases}
    Lw=0 &\text{ in }\co,\\
    w=v &\text{ on }\partial \co.
  \end{cases}
\end{align*}
Then $w\in C^{2,\alpha}_{\loc}(\co)\cap C(\overline\co)$. Thus, $v_0:= v-w$ is a viscosity solution of
\begin{equation}
  \max\{L v_0 -f, v_0- \bar g \}=0 \text{ in } \co,\quad
  v_0=0 \text{ on }\partial\co,
\end{equation}
where $\bar g=g-w$. Then $\bar g=g-v \ge 0$ on $\partial \co$. Take $\bar g^\ve=g^\ve -w\in C^2(\co)\cap C(\ob)$, satisfying $L\bar g^\ve=L g^\ve\ge -M$ in $\co$.  All the other conditions of Theorem~\ref{thm:optimal stopping regularity} are easily verified. So we have $v_0\in W^{2, p}(\co)$ and $v=v_0+w\in W^{2,p}_{\loc}(\co)$. But since $\co\subset \Rn$ is arbitrary, we have $v\in W^{2,p}(\co)$ for any $\co\subset \Rn$.
\end{proof}

\end{document}